\documentclass{amsart}
\usepackage{amsmath}
\usepackage{amssymb}
\usepackage{epsfig}
\usepackage{graphicx}

\newtheorem{theorem}{Theorem}[section]
\newtheorem{lemma}[theorem]{Lemma}
\newtheorem{proposition}[theorem]{Proposition}

\theoremstyle{definition}
\newtheorem{definition}[theorem]{Definition}

\theoremstyle{remark}

\reversemarginpar

\begin{document}

\title[Modified {N}\"{o}rlund polynomials]
{Modified {N}\"{o}rlund polynomials}

\author[A. Dixit]{Atul Dixit}
\address{Department of Mathematics,
Tulane University, New Orleans, LA 70118}
\email{adixit@tulane.edu}

\author[A. Kabza]{Adam Kabza}
\address{Department of Mathematics,
Tulane University, New Orleans, LA 70118}
\email{akabza@tulane.edu}

\author[V. Moll]{Victor H. Moll}
\address{Department of Mathematics,
Tulane University, New Orleans, LA 70118}
\email{vhm@tulane.edu}

\author[C. Vignat]{Christophe Vignat}
\address{Department of Mathematics,
Tulane University, New Orleans, LA 70118 and \\
Dept. of Physics, Universite Orsay Paris Sud, L.~S.~S./Supelec, France}
\email{cvignat@tulane.edu}


\date{\today}


\begin{abstract}
The modified Bernoulli numbers $B_{n}^{*}$ considered by Zagier are generalized to modified N\"{o}rlund polynomials 
${B_{n}^{(\ell)*}}$. For $\ell\in\mathbb{N}$, an explicit expression for the generating function for these polynomials is 
obtained. Evaluations of some spectacular integrals involving Chebyshev polynomials, and of a finite sum involving 
 integrals of the Hurwitz zeta function are also obtained. New results about the $\ell$-fold convolution of the 
 square hyperbolic secant distribution are obtained, such as a differential-difference equation satisfied by a logarithmic 
 moment and a closed-form expression in terms of the Barnes zeta function. 
\end{abstract}

\maketitle

\newcommand{\nn}{\nonumber}
\newcommand{\ba}{\begin{eqnarray}}
\newcommand{\ea}{\end{eqnarray}}
\newcommand{\ift}{\int_{0}^{\infty}}
\newcommand{\ione}{\int_{0}^{1}}
\newcommand{\ifft}{\int_{- \infty}^{\infty}}
\newcommand{\B}{\mathfrak{B}}
\newcommand{\no}{\noindent}
\newcommand{\realpart}{\mathop{\rm Re}\nolimits}
\newcommand{\imagpart}{\mathop{\rm Im}\nolimits}
\newcommand{\norb}{{\mathcal{B}}_{n}^{(\alpha)} }
\newcommand{\lp}{\log \sqrt{2 \pi}}

\newtheorem{Definition}{\bf Definition}[section]
\newtheorem{Thm}[Definition]{\bf Theorem} 
\newtheorem{Example}[Definition]{\bf Example} 
\newtheorem{Lem}[Definition]{\bf Lemma} 
\newtheorem{Note}[Definition]{\bf Note} 
\newtheorem{Cor}[Definition]{\bf Corollary} 
\newtheorem{corollary}[Definition]{\bf Corollary} 
\newtheorem{note}[Definition]{\bf Note} 
\newtheorem{Prop}[Definition]{\bf Proposition} 
\newtheorem{Problem}[Definition]{\bf Problem} 
\newtheorem{Conj}[Definition]{\bf Conjecture} 
\newtheorem{Remark}[Definition]{\bf Remark} 
\newtheorem{Notation}[Definition]{\bf Notation} 
\numberwithin{equation}{section}

\section{Introduction} \label{sec-intro}
\setcounter{equation}{0}

The Bernoulli numbers,  defined by the generating function 
\begin{equation}
\sum_{n=0}^{\infty} B_{n} \frac{z^{n}}{n!}  = \frac{z}{e^{z}-1},
\label{bn-def}
\end{equation}
were extended by N.~E.~N\"{o}rlund \cite[Ch. 6]{norlund-1924a} to
\begin{equation}
\sum_{n=0}^{\infty} B_{n}^{(\alpha)} \frac{z^{n}}{n!} = \left( \frac{z}{e^{z}-1} \right)^{\alpha}.
\label{genfun-nor1}
\end{equation}
\no
Here $\alpha \in \mathbb{C}$.  The coefficients $B_{n}^{(\alpha)}$ are called the 
\textit{N\"{o}rlund polynomials} (these are indeed polynomials in $\alpha$). The list  
$\{ B_{n}^{(\alpha)}: \, n \geq 0 \}$ begins with 
\begin{equation}
\left\{ 1, \, - \frac{\alpha}{2}, \, \frac{1}{12} \alpha(3 \alpha - 1), \, - \frac{1}{8} \alpha^{2} (\alpha- 1), \, 
\frac{1}{240} \alpha( 15\alpha^{3} - 30 \alpha^{2} + 5 \alpha + 2) \right\}.
\end{equation}

For $\alpha \in \mathbb{N}$, the N\"{o}rlund polynomials are expressed as 
the $\alpha$-fold convolutions of 
Bernoulli numbers. This follows from the recurrence 
\begin{equation}
B_{n}^{(\alpha)} = \sum_{j=0}^{n} \binom{n}{j}B_{j}^{(\alpha-1)}B_{n-j}, \quad \text{ for } \alpha \geq 2,
\end{equation}
\no 
obtained from \eqref{genfun-nor1}, and the initial condition $B_{n}^{(1)} = B_{n}$.

Zagier \cite{zagier-1998a} introduced a modification of the Bernoulli numbers via
\begin{equation}
B_{n}^{*} = \sum_{r=0}^{n} \binom{n+r}{2r} \frac{B_{r}}{n+r}, \quad n \in \mathbb{N},
\label{zag-mod-000}
\end{equation}
\no
and their polynomial version 
\begin{equation}
B_{n}^{*} (x)= \sum_{r=0}^{n} \binom{n+r}{2r} \frac{B_{r}(x)}{n+r},
\end{equation}
\no
was studied in detail in \cite{dixit-2014a}. Here $B_{n}(x)$ is the Bernoulli polynomial with 
the generating function 
\begin{equation}
\sum_{n=0}^{\infty} B_{n}(x) \frac{z^{n}}{n!} = \frac{ze^{xz}}{e^{z}-1},
\end{equation}
\noindent
and so along with \eqref{bn-def}, we have $B_{n} = B_{n}(0)$.

\smallskip
\noindent
In particular, \cite{dixit-2014a} establishes the formula 
\begin{equation}
\label{alphais1}
\sum_{n=1}^{\infty} B_{n}^{*}(x) z^{n}  = - \frac{1}{2} \log z - 
\frac{1}{2} \psi( z + 1/z -1 - x )
\end{equation}
\noindent
for the generating function of the Zagier polynomials $B_{n}^{*}(x)$, viewed as a formal power series. Here 
\begin{equation}
\psi(x) = \frac{\Gamma'(x)}{\Gamma(x)}
\end{equation}
\noindent
is the digamma function.  The special case $x=0$ yields 
\begin{equation}
\label{alphais2}
\sum_{n=1}^{\infty} B_{n}^{*} z^{n}  = - \frac{1}{2} \log z - 
\frac{1}{2} \psi( z + 1/z -1  ).
\end{equation}

In the present work, the N\"{o}rlund polynomials are modified in  a similar way as Zagier's.
These \textit{modified N\"{o}rlund polynomials} are defined here by 
\begin{equation}
B_{n}^{ (\alpha)*} := \sum_{r=0}^{n} \binom{n+r}{2r} \frac{B_{r}^{(\alpha)}}{n+r},\hspace{2mm}n\in\mathbb{N}.
\end{equation}
\noindent
The Zagier modification of the Bernoulli numbers \eqref{zag-mod-000} is  the special case $\alpha = 1$.  For $\alpha \in \mathbb{N}$, the 
main result of this paper is an expression
 for the  generating 
function
\begin{equation}
F_{B^{*}}(z;\alpha) =  \sum_{n=1}^{\infty} B_{n}^{(\alpha)*}z^{n}, 
\end{equation}
\noindent
involving derivatives of the digamma function as given in Theorem \ref{gen-fun-modber-0}. This is a
 generalization of \eqref{alphais2}.

\medskip 

\noindent
\textbf{Notation}.  Standard notation is used throughout the paper.  

\medskip

\noindent
1) The generalized binomial coefficients are defined by $$\begin{displaystyle} \binom{x}{n}  = \frac{1}{n!}x(x-1) \cdots (x-n+1) \end{displaystyle},$$ for 
$x \in \mathbb{R}$ and $n \in \mathbb{N}$.
 
 \smallskip

\noindent
2) The harmonic numbers are defined by $$\begin{displaystyle} H_{n} = 1 + \frac{1}{2} + \cdots + \frac{1}{n} \end{displaystyle}. $$

\smallskip

\noindent
3) The gamma function is defined by the integral representation 
$$ \begin{displaystyle} \Gamma(z) = \int_{0}^{\infty} e^{-t}t^{z-1} \, dt
\end{displaystyle}$$ 
\noindent
for $\realpart{z}>0$ and extended by analytic continuation. It satisfies the functional equation $\Gamma(z+1) = z \Gamma(z)$.

\smallskip

\noindent
4) The digamma function $\psi(z)$ is defined by 
$$ \begin{displaystyle} \psi(z) = \frac{d}{dz} \log \Gamma(z) \end{displaystyle}.$$
\noindent
It satisfies $$\psi(z+1) = \psi(z) + \frac{1}{z}.$$

\smallskip

\noindent
5) The Chebyshev polynomials of the first and second kind are defined, respectively, by their Binet representations
\cite{spanier-1987a}
\begin{equation}
 T_{n}(x) = \frac{1}{2} \left[ ( x + \sqrt{x^{2}-1} )^{n} + (x - \sqrt{x^{2}-1})^{n} \right] 
 \label{cheby-t-binet}
 \end{equation}
\noindent
and 
\begin{equation}
U_{n}(x) = \frac{1}{2 \sqrt{x^{2}-1}} \left[ ( x + \sqrt{x^{2}-1} )^{n} - (x - \sqrt{x^{2}-1})^{n} \right]. 
\label{cheby-u-binet}
\end{equation}

\noindent

%

\bigskip

The  work presented here is based on the symbolic notation
\begin{equation}
\label{formal-B}
g(x  + B) = \frac{\pi}{2} \int_{-\infty}^{\infty} g \left( x  - \tfrac{1}{2} + iv \right) \text{sech}^{2}(\pi v) \, dv.
\end{equation}
\noindent

The formula \eqref{formal-B} is based on the fact that, if $L_{B}$ is a random variable with the square secant hyperbolic 
distribution 
\begin{equation}
\rho(x) = \frac{\pi}{2} \text{sech}^{2}( \pi x),
\label{den-sec1}
\end{equation}
\noindent
then 
\begin{equation}
B_{n} = \mathbb{E} \left( \imath L_{B} - \tfrac{1}{2} \right)^{n} = 
\int_{-\infty}^{\infty} \rho(u) \left( \imath u - \tfrac{1}{2} \right)^{n} \, du
\end{equation}
\noindent
so that, symbolically, with $g(x) = x^{n}$, 
\begin{equation}
B_{n} = g(B).
\end{equation}
\noindent
This extends to Bernoulli polynomials as 
\begin{equation}
B_{n}(x) = (B + x)^{n} = \mathbb{E} \left( x + \imath L_{B} - \tfrac{1}{2} \right)^{n}
\end{equation}
\noindent
and to any analytic function $g$ as 
\begin{equation}
  \mathbb{E} \left[ g( x - \tfrac{1}{2} + i L_{B} ) \right]  =  
   \frac{\pi}{2} \int_{-\infty}^{\infty} g\left( x - \tfrac{1}{2} + i v \right) \text{sech}^{2}(\pi v) \, dv.\label{expect-op}
\end{equation}
\noindent
This is complemented with the notation
\begin{equation}
\label{formal-U}
f(x+U) = \int_{0}^{1} f(x+u) \, du,
\end{equation}
that corresponds to  the average over a uniform distribution $U$ on $[0, \, 1]$.

\smallskip

The symbolic form \eqref{formal-B} is a restatement of the umbral approach described in \cite{dixit-2014a}.  The classical
umbral calculus begins with a sequence $\{ a_{n} \}$ and formally transforms it into powers $\mathfrak{a}^{n}$ of a 
new variable $\mathfrak{a}$, named the umbra of $\{ a_{n} \}$. The original sequence is then recovered by the 
evaluation map $\text{eval} \{ \mathfrak{a}^{n} \} = a_{n}$.  The \textit{Bernoulli umbra}, studied in 
\cite{dixit-2014a}, is defined by the generating function 
\begin{equation}
\text{eval} \{ \exp \left( t \mathfrak{B}(x) \right) \} = \frac{te^{xt}}{e^{t}-1}
\end{equation}
\noindent
and it  satisfies, with $\mathfrak{B} = \mathfrak{B}(0)$, 
\begin{equation}
- \mathfrak{B} = \mathfrak{B} + 1, \hspace{1mm}\text{and}\hspace{1mm}(-\mathfrak{B})^n=\mathfrak{B}^n\hspace{1mm}\text{for}\hspace{1mm}n\neq 1,
\end{equation}
\noindent
and 
\begin{equation}
\text{eval} \{ \mathfrak{B}(x) \} = \text{eval} \{ x + \mathfrak{B} \}.
\end{equation}
For more properties of Bernoulli umbrae, the reader is referred to Gessel \cite{gessel-2003a}.
\noindent
Theorem 2.3 in \cite{dixit-2014a} states that the Bernoulli umbra coincides with a random variable 
$i L_{B} - \tfrac{1}{2}$, in the sense that,
\begin{equation}
\text{eval} \{ g( \mathfrak{B}+ x ) \} = \mathbb{E} \left[ g( x - \tfrac{1}{2} + iL_{B} ) \right],    \label{umbral}
\end{equation}
\noindent 
for all admissible functions $g$.

Thus from \eqref{formal-B}, \eqref{expect-op} and \eqref{umbral}, one obtains the three equivalent notations
\begin{equation}
g(x + B) =  \mathbb{E} \left[ g( x - \tfrac{1}{2} + iL_{B} ) \right]=\text{eval} \{ g( \mathfrak{B}(x) ) \},
\end{equation}
and for brevity, we will mostly use the symbolic form $g(x+B)$.

Now take $\ell$ independent 
copies $\{ L_{B_{1}}, \cdots, L_{B_{\ell}} \}$ of the random variable $L_{B}$.  The  density $\rho_{\ell}(x)$ associated to 
$\begin{displaystyle} L = L_{B_{1}} + \cdots L_{B_{\ell}} \end{displaystyle}$
is then the $\ell$-fold  convolution of the density $\rho(x)$ of each summand. This is computed by the  recurrence 
\begin{equation}
\rho_{\ell}(x) = \int_{-\infty}^{\infty} \rho_{\ell-1}(u) \rho_{1}(x-u) \, du,
\label{convo-1}
\end{equation}
\noindent
starting with $\rho_{1}(x) = \rho(x)$.  A direct computation of the densities $\{ \rho_{\ell} \}$ is remarkably difficult.  The 
case $\ell = 2$ is presented in Section \ref{sec-family}.  The case $\ell=1$ and $g(x)=\log x$ of the  formula 
\begin{equation}
\label{sum-umbra}
\mathbb{E}[g(x - \tfrac{\ell}{2} + iL_{B_{1}} + i L_{B_{2}} + \cdots + i L_{B_{\ell}}  )] =  \int_{-\infty}^{\infty} 
\rho_{\ell}(u) g ( x - \tfrac{\ell}{2} + i u)\, du
\end{equation}
\noindent
was used in \cite{dixit-2014a} to evaluate the generating functions of the modified Bernoulli numbers and of 
Zagier polynomials. In the umbral notation, this quantity can be written as 
\begin{equation}
\text{eval} \left[ g( x + \mathfrak{B}_{1} + \cdots  + \mathfrak{B}_{\ell} ) \right] = 
\text{eval} \left[ g \left( \mathfrak{B}^{(\ell)}(x) \right) \right],
\end{equation}
so that the umbra associated with the modified N\"{o}rlund polynomials is 
\begin{equation}
\mathfrak{B}^{(\ell)} = \mathfrak{B}_{1} + \mathfrak{B}_{2} + \cdots + \mathfrak{B}_{\ell}.
\end{equation}

This is extended here to compute the corresponding generating function of the 
modified N\"{o}rlund polynomials.   A crucial step in the argument uses the following result which evaluates 
the logarithm of the umbra $\mathfrak{B}^{(\ell)}$. 

\begin{theorem}
\label{eval-mess00}
Let $\ell \in \mathbb{N}$ be fixed. For $x \in \mathbb{R}$,
\begin{equation*}
\text{eval} \{ \log \mathfrak{B}^{(\ell)}(x) \} 
=  - H_{\ell-1} + 
\frac{d^{\ell-1}}{dx^{\ell-1}} \left\{ \binom{x-1}{\ell -1} \psi \left( x - \left\lfloor \frac{\ell}{2} \right\rfloor \right) \right\}.
\end{equation*}
\noindent
Here $H_{r}$ is the harmonic number and $\psi(x)$ is the digamma function and $ \left\lfloor \frac{\ell}{2} \right\rfloor$ denotes 
the floor function. 
\end{theorem}

The following generating function for the modified N\"{o}rlund polynomials is now obtained from the 
previous theorem.

\begin{theorem}
\label{gen-fun-modber-0}
Let $\ell \in \mathbb{N}$ be fixed. The generating function 
\begin{equation*}
F_{B^{*}}(z;\ell)  =   \sum_{n=1}^{\infty} B_{n}^{(\ell)*}z^{n}
\end{equation*}
\noindent
for the modified N\"{o}rlund polynomials $B_{n}^{(\ell)*}$ is given by 
\begin{eqnarray*}
F_{B^{*}}(z;\ell)  =  - \frac{1}{2} \log z  
 + \frac{1}{2} \left[ H_{\ell-1} - \frac{d^{\ell-1}}{dx^{\ell-1}} \left\{ 
\binom{x-1}{\ell - 1} \psi\left( x - \left\lfloor \frac{\ell}{2} \right\rfloor \right) \right\}  \right]
\end{eqnarray*}
\noindent 
evaluated at $x = z + 1/z +\ell-2$.
\end{theorem}

An alternate representation for $\text{eval} \{  \log \mathfrak{B}^{(\ell)}(x) \}$  gives the following 
remarkable integral evaluation involving the density $\rho_{\ell}(x)$.
\begin{theorem}
\label{coro-nice1}
Let $\ell \in \mathbb{N}$ be fixed. Then 
\begin{multline*}
 \int_{0}^{\infty}  \log(1+ bu^{2}) \, \rho_{\ell}(u) \, du =  \\
 -    \log \left| x - \frac{\ell}{2} \right| -  H_{\ell-1} + \frac{d^{\ell-1}}{dx^{\ell-1}} \left\{ 
\binom{x-1}{\ell-1} \psi\left( x - \left\lfloor \frac{\ell}{2} \right\rfloor \right) \right\}  
\end{multline*}
\noindent
with $b = (x - \ell/2)^{-2}$.
\end{theorem}


Theorems \ref{gen-fun-modber-0} and \ref{coro-nice1} readily give the following result. We record it as 
a theorem only to emphasize the link between the generating function of the modified N\"{o}rlund polynomials
and the definite integral containing the density $\rho_{\ell}(x)$.
\begin{theorem}
\label{gen-fun-density}
Let $\ell \in \mathbb{N}$ be fixed. The generating function 
\begin{equation*}
F_{B^{*}}(z;\ell)  =   \sum_{n=1}^{\infty} B_{n}^{(\ell)*}z^{n}
\end{equation*}
\noindent
for the modified N\"{o}rlund polynomials $B_{n}^{(\ell)*}$ is given by 
\begin{eqnarray*}
F_{B^{*}}(z;\ell)  =  - \frac{1}{2} \log z  - \frac{1}{2} \log \left| x - \frac{\ell}{2} \right| - 
\frac{1}{2} \int_{0}^{\infty} \log(1 + bu^{2}) \rho_{\ell}(u) \, du 
\end{eqnarray*}
\noindent 
with $b = (x - \ell/2)^{-2}$ and  $x = z + 1/z +\ell-2$.
\end{theorem}

Section \ref{sec-gf1} describes a symbolic formalism based on two probability densities. This is used in Section \ref{sec-bf-nor}
to obtain an expression for the generating function of the N\"{o}rlund polynomials. Section \ref{sec-family} presents a 
family of densities that provide an alternative form of this generating function. These densities satisfy a differential-difference 
equation and the initial conditions are evaluated in Section \ref{sec-gf-nor00}. The last two sections uses the densities 
described above to evaluate
some definite integrals involving Chebyshev polynomials and the Hurwitz zeta function.  A direct evaluation of these 
examples seems out of the range of the current techniques of integration. 

\section{Some symbolic formalism}
\label{sec-gf1}

The definition of the digamma function as $\psi(z) = \frac{d}{dz} \log \Gamma(z)$  immediately gives the 
evaluation 
\begin{equation}
\int_{0}^{1} \psi(x+t) \, dt = \log x.
\label{rela-1}
\end{equation}
\noindent
The inversion formula
\begin{equation}
\psi(x) = \frac{\pi}{2} \int_{-\infty}^{\infty} \log(x - \tfrac{1}{2} + i u) \, \text{sech}^{2} \pi u \, du
\label{rela-2}
\end{equation}
\noindent
was established in Theorem $2.5$ of  \cite{dixit-2014a}.

In the  notation \eqref{formal-U} and \eqref{formal-B},  \eqref{rela-1} and \eqref{rela-2} 
expresses the equivalence of the relations
\begin{equation}
\psi(x+U) = \log x \text{ and }  \psi(x) = \log(x+B).
\label{inversion-1}
\end{equation}
\noindent
This is now shown to be a particular case  of a more general inversion formula.

\begin{definition}
For real-valued functions $f$ and $g$, define in recursive form
\begin{equation}
f(x+ U^{(\ell)}) = f(x+ U + U^{(\ell-1)}) \quad \text{ for } \ell \geq 2,
\end{equation}
\noindent
with $U^{(1)} = U$, and similarly 
\begin{equation}
g(x+ B^{(\ell)}) = g(x+ B + B^{(\ell-1)})  \quad \text{ for } \ell \geq 2,
\end{equation}
\noindent
with $B^{(1)} = B$.
\end{definition}

In the lemma given below, this  symbolic formalism is connected to anti-derivatives $F^{(\ell)}$ of 
the function $f$, defined as any function $F^{(\ell)}$ such that 
 $\begin{displaystyle} \frac{d^{\ell}}{dx^{\ell}} F^{(\ell)}(x) = f(x) \end{displaystyle}$,
via the classical forward difference operator $\Delta$ defined by 
\begin{equation}
\Delta f(x) = f(x+1)-f(x).
\end{equation}
It is clear that if $f$ is a polynomial of degree $\ell$, then $\Delta f(x)$ is also a polynomial and its 
degree is $\ell-1$.

\begin{lemma}
\label{del-anti}
Let $F^{(\ell)}$ be an antiderivative of $f$ of order $\ell$. Then 
\begin{equation}
f( x + U^{(\ell)}) =  \Delta^{\ell} F^{(\ell)}(x).
\end{equation}
\end{lemma}
\begin{proof}
The case $\ell=1$ is straightforward since 
\begin{equation}
f(x+U) = \int_{0}^{1} f(x+u) \, du = F(x+1) - F(x)
\end{equation}
\noindent
by the Fundamental Theorem of Calculus. The inductive step is 
\begin{eqnarray*}
f \left(  x + U^{(\ell+1)} \right) & = & f \left( x + U^{(\ell)} + U \right) \\
 & = & \Delta^{\ell} F^{(\ell)}(x+U) \\
 & = & \Delta^{\ell} \left[ F^{(\ell+1)}(x+1) - F^{(\ell+1)}(x) \right] \\
 & = & \Delta^{\ell+1} F^{(\ell+1)}(x).
 \end{eqnarray*}
\end{proof}

The next result  is a generalization of \eqref{inversion-1}: it shows that the symbols  $U$ and $B$ invert 
each other. The proof uses the evaluation of the definite integral 
\begin{equation}
\int_{0}^{\infty} \frac{\cos zv}{\cosh^{2} \pi v} \, dv = \frac{z}{2 \pi} \, \left(  \text{sinh} \frac{z}{2} \right)^{-1},
\label{39822a}
\end{equation}
which is obtained from entry $3.982.2$ in \cite{gradshteyn-2015a}:
\begin{equation}
\label{39822b}
\int_{0}^{\infty} \frac{\cos ax}{\cosh^{2} \beta x} \, dx = \frac{\pi a}{2 \beta^{2}} \, \left(  \text{sinh} \frac{\pi a}{2 \beta} \right)^{-1}, 
\quad \realpart{\beta}>0, \, a > 0.
\end{equation}
\noindent 
A proof of this entry will appear in \cite{dixit-2014f}.

\begin{theorem}
For any admissible formal power series,
\begin{equation}
g(x) = f(x + U)  \text{ is equivalent to } f(x) = g(x+B).
\end{equation}
\end{theorem}
\begin{proof}
In view of linearity, it suffices to consider the case $f(x) = x^{n}$.  Start with the generating function 
\begin{equation*}
\sum_{n=0}^{\infty} \frac{(x+U)^{n}}{n!} z^{n} =   e^{z(x+U)}   
  =  \int_{0}^{1} e^{z(x+u)} du 
  =  e^{zx} \frac{e^{z}-1}{z},
\end{equation*}
\noindent
and 
\begin{eqnarray*}
\sum_{n=0}^{\infty} \frac{(x+B)^{n}}{n!} z^{n} & = & e^{z(x+B)} = \frac{\pi}{2} 
\int_{-\infty}^{\infty} e^{z \left( x+ iv - \tfrac{1}{2} \right)} \text{sech}^{2}(\pi v) \, dv \\
& = & \frac{\pi}{2} e^{z \left( x - \tfrac{1}{2} \right) } \int_{-\infty}^{\infty} e^{izv} \text{sech}^{2}(\pi v) \, dv \\
& = & \pi e^{z \left( x - \tfrac{1}{2} \right) } \int_{0}^{\infty} \frac{ \cos(zv) }{\cosh^{2}(\pi v)}  \, dv.
\end{eqnarray*}
\noindent
The evaluation \eqref{39822b} gives 
\begin{equation}
\sum_{n=0}^{\infty} \frac{(x+B)^{n}}{n!} z^{n} = e^{zx} \frac{z}{e^{z}-1}.
\end{equation}
\noindent
Now assume first that $g(x) = f(x+U)$, i.e., $g(x) = (x+U)^{n}$. Then 
\begin{eqnarray*}
\sum_{n=0}^{\infty} \frac{g(x+B) z^{n}}{n!} & = & 
\sum_{n=0}^{\infty} \frac{(x+B+U)^{n}}{n!} z^{n} \\ 
& = & e^{z(x+B)} \frac{e^{z}-1}{z} \\
& = & e^{zx} \frac{z}{e^{z}-1} \frac{e^{z}-1}{z} = e^{zx}.
\end{eqnarray*}
\noindent
From here it follows that $(x+B+U)^{n} = x^{n}$.  The other direction is established in a similar form.
\end{proof}

\begin{note}
A direct extension gives the equivalence of the statements 
\begin{equation}
g(x) = f( x + U^{(\ell)}) \text{ and }f(x) = g(x + B^{(\ell)}), \quad \text{ for all } \ell \in \mathbb{N},
\label{inversion-11}
\end{equation}
\noindent
which can be proved by induction. 
\end{note}

\section{The generating function of the modified N\"{o}rlund polynomials}
\label{sec-bf-nor}

This  section  uses the results of the previous 
section to prove  an expression for the 
horizontal generating function of the modified N\"{o}rlund polynomials $B_{n}^{(\ell)*}$ as a 
formal power series. 
Here $\ell$ is a fixed positive integer. This generating function is defined by 
\begin{equation}
F_{B^{*}}(z;\ell) = \sum_{n=1}^{\infty} B_{n}^{(\ell)*}z^{n}.
\end{equation}

\begin{lemma}
\label{delta-1}
Let $\psi(x)$ be the digamma function and $H_{\ell}$ the $\ell$-th harmonic number. Then for $\ell \geq 1$ and $-1 \leq p \leq \ell-1$, 
\begin{equation}
\Delta^{\ell} \left[ \binom{x+p}{\ell} \psi(x) \right] = H_{\ell} + \psi(x+p+1).
\end{equation}
\end{lemma}
\begin{proof}
The result is established first for $p=0$. Define 
\begin{equation}
h_{\ell}(x) = \Delta^{\ell} \left[ \binom{x}{\ell} \psi(x) \right]
\end{equation}
\noindent
and observe that 
\begin{eqnarray*}
h_{\ell+1}(x) - h_{\ell}(x) & = & \Delta^{\ell+1} \left[ \binom{x}{\ell+1}  \psi(x) \right] - 
\Delta^{\ell} \left[ \binom{x}{\ell} \psi(x) \right] \\
& = & \Delta^{\ell} \Delta  \left[ \binom{x}{\ell+1}  \psi(x) \right]  - \Delta^{\ell} \left[ \binom{x}{\ell} \psi(x) \right] \\
& = & \Delta^{\ell} \left[ \binom{x+1}{\ell+1} \psi(x+1) - \binom{x}{\ell+1} \psi(x) - \binom{x}{\ell} \psi(x) \right].
\end{eqnarray*}
\noindent
The identity 
\begin{equation}
\binom{x+1}{\ell+1} = \binom{x}{\ell+1} + \binom{x}{\ell} 
\end{equation}
\noindent 
gives 
\begin{eqnarray*}
h_{\ell+1}(x) - h_{\ell}(x) & = & \Delta^{\ell} \left[ \binom{x}{\ell+1} \left( \psi(x+1) - \psi(x) \right) + 
\binom{x}{\ell} \left( \psi(x+1) - \psi(x) \right) \right] \\
& = & \Delta^{\ell} \left[  \frac{1}{x} \left( \binom{x}{\ell+1} + \binom{x}{\ell} \right) \right] \\
& = & \Delta^{\ell} \left[ \frac{(x-1) \cdots (x-\ell)}{(\ell+1)!} + 
\frac{(x-1) \cdots (x-\ell+1)}{ \ell!} \right]
\end{eqnarray*}
\noindent
The second fraction is a polynomial in $x$ of degree $\ell-1$. Therefore $\Delta^{\ell}$ annihilates it. The 
first fraction is a polynomial of degree $\ell$ and only its leading term survives the application of  $\Delta^{\ell}$. This leads to 
the difference equation
\begin{equation}
h_{\ell+1}(x) - h_{\ell}(x) = \Delta^{\ell} \frac{x^{\ell}}{(\ell+1)!} = \frac{1}{\ell+1},
\label{recu-110}
\end{equation}
\noindent
since $\Delta^{\ell} x^{\ell} = \ell!$. The latter follows directly from Lemma \ref{del-anti}:  indeed,  choosing $f(x) = 1$ produces 
$F^{(\ell)}(x) = x^{\ell}/\ell!$ and therefore 
\begin{equation}
\Delta^{\ell} \frac{x^{\ell}}{\ell!} = f( x + U^{(\ell-1)}) = 1,
\end{equation}
\noindent 
which gives the result. Now write \eqref{recu-110} as 
\begin{equation}
h_{\ell+1}(x) - \psi(\ell+2) = h_{\ell}(x) - \psi(\ell+1),
\end{equation}
\noindent
so that 
\begin{equation}
h_{\ell}(x) = h_{1}(x) + \psi(\ell+1) - \psi(2).
\label{form-final-h}
\end{equation}
\noindent
Now
\begin{equation}
h_{1}(x) = \Delta \left[ \binom{x}{1} \psi(x) \right]  = 1 + \psi(x+1)\label{h1}
\end{equation}
\noindent
gives the stated result for $p=0$. 

Now assume $p \neq 0$ and that $1 \leq p \leq \ell-1$. Observe that 
\begin{eqnarray*}
\Delta^{\ell} \left[ \binom{x+p}{\ell} \left( \psi(x+p) - \psi(x) \right) \right] & = & 
\Delta^{\ell} \left[ \binom{x+p}{\ell} \left( \frac{1}{x+p-1} + \cdots + \frac{1}{x} \right) \right] \\
& = & \frac{1}{\ell!} \Delta^{\ell} \left[ \prod_{u=0}^{\ell-1} (x+ p-u) \times  \left( \frac{1}{x+p-1} + \cdots + \frac{1}{x} \right)       \right]
\end{eqnarray*}
\noindent
The bounds $1 \leq p \leq \ell-1$ show that the last expression is actually a polynomial of degree $\ell-1$. One can also easily check that when $p=-1$, $\binom{x+p}{\ell} \left( \psi(x+p) - \psi(x) \right)$ is a polynomial of degree $\ell-1$. This 
implies that for $-1 \leq p \leq \ell-1, p\neq 0$,
\begin{equation}
\Delta^{\ell} \left[ \binom{x+p}{\ell} \left( \psi(x+p) - \psi(x) \right) \right] = 0.
\end{equation}
\noindent
It follows that 
\begin{eqnarray*}
\Delta^{\ell} \left[ \binom{x+p}{\ell} \psi(x) \right] & = & \Delta^{\ell} \left[ \binom{x+p}{\ell} \psi(x+p) \right]  \\ 
& = & h_{\ell}(x+p) \\
& = & H_{\ell} + \psi(x+p+1),
\end{eqnarray*}
\noindent
as can be seen from \eqref{form-final-h} and \eqref{h1}. This completes the argument.
\end{proof}

The proof of Theorem \ref{eval-mess00} is given next.

\begin{proof}
Using the symbolic operator $B$, the left-hand side of Theorem \ref{eval-mess00} can be written as $\log(x+B^{(\ell)})$. Let $f(x)$ denote
the right-hand side of Theorem \ref{eval-mess00}, i.e.,
\begin{align}
f(x) =-H_{\ell-1} +  \frac{d^{\ell-1}}{dx^{\ell-1}}  \left[ \binom{x-1}{\ell-1} \psi\left(x-\left\lfloor \frac{\ell}{2} \right\rfloor\right) \right].
\label{def-ff}
\end{align}
Using \eqref{inversion-11}, it suffices to prove
\begin{equation*}
f\left(x+U^{(\ell-1)}\right)=\log\left(x+B\right).
\end{equation*}
However from  \eqref{inversion-1},
\begin{equation}
\log \left( x+ B \right)= \psi \left( x \right).
\end{equation}
\noindent 
So we only need to show that
\begin{equation*}
f\left(x+U^{(\ell-1)}\right)=\psi(x).
\end{equation*}
Now Lemma \ref{del-anti} gives
\begin{equation}
f( x + U^{(\ell-1)}) = \Delta^{\ell-1} F^{(\ell-1)}(x),
\label{form-314}
\end{equation}
\noindent
and writing $f(x)$ as 
\begin{equation}
f(x) = \frac{d^{\ell-1}}{dx^{\ell-1}} \left[ \binom{x-1}{\ell-1} \psi\left(x-\left\lfloor \frac{\ell}{2} \right\rfloor\right) - \frac{x^{\ell-1}}{(\ell-1)!} H_{\ell-1} \right]
\end{equation}
\noindent
produces 
\begin{equation}
F^{(\ell-1)}(x) = \binom{x-1}{\ell-1} \psi\left(x-\left\lfloor \frac{\ell}{2} \right\rfloor\right) - \frac{x^{\ell-1}}{(\ell-1)!} H_{\ell-1}.
\end{equation}
\noindent
Then 
\begin{equation}
\Delta^{\ell-1} F^{(\ell-1)}(x) = -H_{\ell-1} + \Delta^{\ell-1} 
\left[ \binom{x-1}{\ell-1} \psi\left(x-\left\lfloor \frac{\ell}{2} \right\rfloor\right)\right]
\end{equation}
\noindent
and \eqref{form-314} gives
\begin{equation}
f(x + U^{(\ell-1)} ) = - H_{\ell-1} + \Delta^{\ell-1} \left[ \binom{x-1}{\ell-1} \psi\left(x-\left\lfloor \frac{\ell}{2} \right\rfloor\right)\right].
\end{equation}
\noindent 
Now Lemma \ref{delta-1}, with $\ell$ replaced by $\ell-1$, $x$ replaced by $x-\left\lfloor \tfrac{\ell}{2} \right\rfloor$ and $p = \lfloor \tfrac{\ell}{2} \rfloor - 1$, yields 
\begin{equation}
f(x + U^{(\ell-1)}) = \psi(x).
\end{equation}
This completes the proof.
\end{proof}

The  proof of Theorem \ref{gen-fun-modber-0}, which expresses  the generating function for the modified N\"{o}rlund 
polynomials,  is now given.

\begin{proof}
The proof of the identity 
\begin{equation}
F_{B^{*}}(z) = \sum_{n=1}^{\infty} B_{n}^{*} z^{n} = - \text{eval} \left\{ \tfrac{1}{2} \log \left((1-z)^{2} - z \mathfrak{B} \right) \right\}
\end{equation}
\noindent
given in \cite[Equation (3.4)]{dixit-2014a} can be adapted  to derive, in a similar manner, the relation
\begin{equation}
F_{B^{*}}(z; \ell)  = \sum_{n=1}^{\infty} B_{n}^{(\ell)*}z^{n} = 
- \text{eval} \left\{ \frac{1}{2} \log \left( (1-z)^{2} - z \mathfrak{B}^{(\ell)} \right) \right\}.
\end{equation}
\noindent
This is described next. Basic facts of umbral calculus, namely $- \mathfrak{B}^{(\ell)} = \mathfrak{B}^{(\ell)}  + \ell$ and 
$x + \mathfrak{B}^{(\ell)} = \mathfrak{B}^{(\ell)}(x)$, give 
\begin{eqnarray}\label{eval-gen}
\sum_{n=1}^{\infty} B_{n}^{(\ell)*} z^{n} & = & - \frac{1}{2} \log z - 
\frac{1}{2} \text{eval} \left\{ \log \left( z + \frac{1}{z} - 2 - \mathfrak{B}^{(\ell)} \right) \right\} \\
& = & - \frac{1}{2} \log z - \frac{1}{2}  \text{eval} \left\{ \log \left( z + \frac{1}{z} - 2 + \mathfrak{B}^{(\ell)} + \ell  \right) \right\} 
\nonumber \\
& = & - \frac{1}{2} \log z - \frac{1}{2}  \text{eval} \left\{ \log  \mathfrak{B}^{(\ell)} \left( z + \frac{1}{z} + \ell  - 2  \right) \right\}.
\nonumber 
\end{eqnarray}
\noindent
The final step uses Theorem \ref{eval-mess00}. 
\end{proof}
The proof of Theorem  \ref{coro-nice1} is presented next.

\noindent
\begin{proof}
Start with 
\begin{eqnarray*}
\text{eval}\{ \log \mathfrak{B}^{( \ell )}(x)\} & = & 
\text{eval}\{ \log \left(  x + \mathfrak{B}_{1} + \cdots + \mathfrak{B}_{\ell} \right)  \}\\
& = & \mathbb{E} \left[ \log \left( x - \tfrac{\ell}{2} + i (L_{B_{1}} + \cdots + L_{B_{\ell}} ) \right) \right].
\end{eqnarray*}
\noindent
Introduce the notation $L = L_{B_{1}} + \cdots + L_{B_{\ell}}$ and since the density $\rho_{\ell}$ is an 
even function, $L$ and $-L$ have the same distribution.  Therefore, with $b = (x - \ell/2)^{-2}$, 
\begin{eqnarray}\label{eval-int}
\text{eval}\{ \log \mathfrak{B}^{( \ell )}(x) \}& = & \frac{1}{2} \mathbb{E} \left[ \log \left( 
\left( x - \frac{\ell}{2} \right)^{2}  + L^{2} \right) \right] \nonumber\\
& = & \log \left| x - \frac{\ell}{2} \right| + \frac{1}{2} \mathbb{E} \left[ \log( 1 + bL^{2} ) \right] \nonumber\\
& = & \log \left| x - \frac{\ell}{2} \right| + \frac{1}{2} \int_{-\infty}^{\infty} \log(1 + bu^{2}) \rho_{\ell}(u) \, du\nonumber\\
& = & \log \left| x - \frac{\ell}{2} \right| + \int_{0}^{\infty} \log(1 + bu^{2}) \rho_{\ell}(u) \, du,
\end{eqnarray}
\noindent
since $\rho_{\ell}$ is an even function of $u$. The result now follows from Theorem \ref{eval-mess00}.
\end{proof}

\section{A family of densities and a differential-difference equation}
\label{sec-family}

This section discusses the densities $\rho_{n}(x)$ defined by the recurrence 
\begin{equation}
\rho_{n}(x) = \int_{-\infty}^{\infty} \rho_{n-1}(y) \rho_{1}(x-y) dy
\end{equation}
\noindent 
with initial condition 
\begin{equation}
\rho_{1}(x) = \rho(x) = \frac{\pi}{2} \text{sech}^{2}(\pi x).
\label{density-s1}
\end{equation}
\noindent
These densities provide the evaluation 
\begin{equation}
\mathbb{E} \left[ g \left(  x - \tfrac{\ell}{2} + i \sum_{j=1}^{\ell} B_{j} \right) \right]  
= \int_{-\infty}^{\infty} g( x - \tfrac{1}{2} + i v) \rho_{\ell}(v) dv.
\end{equation}
\noindent
In particular, the generating function of the N\"{o}rlund polynomials is linked to these densities via Theorem \ref{gen-fun-density}.
Some properties of these densities are described next. 

 \smallskip
 
 \begin{lemma}
 The Fourier transform of $\rho_{1}(x)$ is given by 
 \begin{equation}
  \widehat{\rho_{1}} (\xi) = \frac{\pi \xi}{\sinh \pi \xi}.
 \end{equation}
 \end{lemma}
 \begin{proof}
 The Fourier transform is given by 
 \begin{eqnarray*}
   \widehat{\rho_{1}} (\xi) & = & \int_{-\infty}^{\infty} \frac{\pi}{2} \text{sech}^{2}(\pi x) e^{-2 \pi i x \xi} \, dx \\
& = & \pi \int_{0}^{\infty} \frac{\cos( 2 \pi x \xi)}{\cosh^{2}(\pi x)} \, dx,
\end{eqnarray*}
\noindent 
and the result follows by using \eqref{39822a}.  
 \end{proof}
 
\begin{corollary}
The  Fourier transform of $\rho_{\ell}(x)$ is given by 
\begin{equation}
 \widehat{\rho_{\ell}} (\xi)=  \left( \frac{\pi \xi}{\sinh \pi \xi} \right)^{\ell}.
\end{equation}
\end{corollary}
\begin{proof}
This follows directly from the fact that Fourier transform converts convolutions into products.
\end{proof}

The Fourier inversion formula now gives a representation for the density $\rho_{\ell}(x)$ as 
\begin{equation}
\rho_{\ell}(x) = \frac{1}{\pi} \int_{-\infty}^{\infty}  \left( \frac{y}{\sinh y} \right)^{\ell} e^{2  i x y} \, dy.
\label{form-fn}
\end{equation}

\begin{note}
J.~Pitman and M.~Yor \cite[p.~299]{pitman-2003a} studied the function 
\begin{equation}
\phi_{\ell}(x) = \frac{1}{2 \pi} \int_{-\infty}^{\infty} \left( \frac{y}{\sinh y} \right)^{\ell} e^{i x y } dy
\end{equation}
\noindent
as part of their study on infinitely divisible distributions generated by L\'{e}vy processes associated 
with hyperbolic functions. The expression \eqref{form-fn} shows that 
\begin{equation}
\rho_{\ell}(x) = 2 \phi_{\ell}(2x).
\label{relation-phi}
\end{equation}
\noindent
These authors show that $\phi_{t}$ satisfies the differential-difference equation
\begin{equation}
\label{recu-1}
\ell(\ell+1) \phi_{\ell+2}(x) = (x^{2}+\ell^{2}) \phi_{\ell}''(x) + 2(\ell+2) x \phi_{\ell}'(x) + (\ell+1)(\ell+2) \phi_{\ell}(x).
\end{equation}
\end{note}

\begin{note}
The authors of \cite{pitman-2003a} also consider the transform 
\begin{equation}
\psi_{\ell}(x)  = \frac{1}{2 \pi} \int_{-\infty}^{\infty} \left( \frac{1}{\cosh y} \right)^{\ell} e^{ixy} dy
\end{equation}
\noindent
and prove the explicit formulae 
\begin{equation}
\psi_{\ell}(x) = \frac{2^{\ell-2}}{\pi \Gamma(\ell)} \left| \Gamma \left( \frac{\ell +i x}{2} \right) \right|^{2}.
\label{expl-1}
\end{equation}
\noindent
Then, they state  `\textit{we do not know of any explicit formula for $\phi_{\ell}$ like} \eqref{expl-1}  \textit{valid for 
general} $\ell> 0$'.
\end{note}

\begin{note}
The density functions $\rho_{\ell}(x)$ have also appeared in Airault \cite[p. 2109, (1.52), (1.53)]{airault-2008a}. This author 
proves that 
\begin{eqnarray}
\rho_{2\ell}(x) & = &  \frac{\pi^{1-2 \ell }}{2(2 \ell-1)!} \frac{d^{2 \ell}}{dx^{2 \ell}} \left[ \frac{Q_{2 \ell-1}(\pi x)}{\tanh(\pi x)}  \right]  
\label{airault-1} \\
& & \nonumber \\
\rho_{2 \ell+1}(x) & = & \frac{\pi^{-2 \ell}}{2(2 \ell)!} \frac{d^{2 \ell+1}}{dx^{2 \ell+1}} \left[ Q_{2 \ell}(\pi x) \tanh(\pi x) \right] 
\nonumber 
\end{eqnarray}
\noindent
where 
\begin{equation}
Q_{2 \ell}(x) = \prod_{\begin{stackrel} {1 \leq j \leq 2 \ell-1}{j \text{ odd}} \end{stackrel}} 
\left( x^{2} + \frac{\pi^{2} j^{2}}{4} \right)
\end{equation}
\noindent
and 
\begin{equation}
Q_{2 \ell+1}(x) = x \prod_{1 \leq j \leq \ell} (x^{2} + j^{2} \pi^{2}). 
\end{equation}
\end{note}

The differential-difference equation \eqref{recu-1} produces 
\begin{equation}
\ell(\ell+1) \rho_{\ell+2}(x) = \frac{(4x^{2}+\ell^{2})}{4} \rho_{\ell}''(x) + 2x(\ell+2) \rho_{\ell}'(x) + (\ell+1)(\ell+2)\rho_{\ell}(x),
\label{recu-2}
\end{equation}
\noindent
so that  $\rho_{\ell}(x)$ can be  obtained from \eqref{recu-2} and the initial conditions 
$\rho_{1}(x)$ in \eqref{density-s1} and $\rho_{2}(x)$. Even though the expression for $\rho_2(x)$ is 
well-known \cite[p.~312, Table 6]{pitman-2003a}, it is derived here for the sake of completeness.

\begin{lemma}
 The density function $\rho_{2}$ is given by 
 \begin{equation}
 \rho_{2}(x) = \frac{\pi}{\sinh^{2}(\pi x)} \left( \pi x \, \coth(\pi x) - 1 \right).\label{density-s2}
 \end{equation}
 \end{lemma}
 \begin{proof}
 The relation \eqref{convo-1} gives 
 \begin{eqnarray*}
 \rho_{2}(x) & = & \int_{-\infty}^{\infty} \rho_{1}(u) \rho_{1}(x-u) \, du \\
   & = & \frac{\pi^{2}}{4} \int_{-\infty}^{\infty} \frac{du}{\cosh^{2}(\pi u) \, \cosh^{2}(\pi(x-u))} \\
   & = & \frac{\pi}{4} \int_{-\infty}^{\infty} \frac{dt}{\cosh^{2}t \, \cosh^{2}(\pi x -t)}.
\end{eqnarray*}
\noindent
The change of variable  $w= e^{2t}$ gives 
\begin{equation}
\rho_{2}(x) = 2 \pi \int_{0}^{\infty} \frac{w \, dw}{(w+1)^{2} (\alpha + \beta w )^{2}},
\end{equation}
\noindent
with $\alpha = e^{\pi x}$ and $\beta = e^{- \pi x}$. The final integral is evaluated by partial 
fractions to produce the stated result.
 \end{proof}

The higher densities $\rho_{\ell}(x), \ell>2$, can now be computed via \eqref{recu-2} and the expressions in 
\eqref{density-s1} and \eqref{density-s2}.

\smallskip

The next step is to show that the integral appearing in Theorem \ref{coro-nice1} satisfies a 
differential-difference equation.

\begin{theorem}
The integral 
\begin{equation*}
z_{\ell}(b):=\int_{0}^{\infty}  \log(1+ bu^{2}) \, \rho_{\ell}(u) \, du
\end{equation*}
\noindent 
 satisfies  the differential-difference equation
\begin{align}\label{dde}
\ell(\ell+1)y_{\ell+2}(x+1)&=x(x-\ell)y_{\ell}''(x)+2(\ell+1)\left(x-\frac{\ell}{2}\right)y_{\ell}'(x)\nonumber\\
&\quad \quad \quad +\ell(\ell+1)y_{\ell}(x)+\frac{\ell^2}{4\left(x-\frac{\ell}{2}\right)^2}
\end{align}
for $b=(x-\ell/2)^{-2}$.
\end{theorem}
\begin{proof}
Let $b>0$. Start with \eqref{recu-2}, i.e.,
\begin{equation*}
\ell(\ell+1) \rho_{\ell+2}(u) = \left(u^{2}+\frac{\ell^{2}}{4}\right) \rho_{\ell}''(u) + 2u(\ell+2) \rho_{\ell}'(u) + (\ell+1)(\ell+2)\rho_{\ell}(u).
\end{equation*}
Multiply both sides by $\log(1+bu^2)$ and integrate both sides from $0$ to $\infty$ to obtain
\begin{align}\label{dde-fform}
\ell(\ell+1)z_{\ell+2}(b)&=\int_{0}^{\infty}  \left(u^2+\frac{\ell^2}{4}\right)\log(1+ bu^{2}) \, \rho_{\ell}''(u) \, du\nonumber\\
&\quad+2(\ell+2)\int_{0}^{\infty}  u\log(1+ bu^{2}) \, \rho_{\ell}'(u) \, du+(\ell+1)(\ell+2)z_{\ell}(b).
\end{align}
Let 
\begin{align}
I_{1}(b,\ell)&:=\int_{0}^{\infty}  u\log(1+ bu^{2}) \, \rho_{\ell}'(u) \, du,\nonumber\\
I_{2}(b,\ell)&:=\int_{0}^{\infty}  \left(u^2+\frac{\ell^2}{4}\right)\log(1+ bu^{2}) \, \rho_{\ell}''(u) \, du.
\end{align}
Consider $I_{1}(b,\ell)$ first. Integration by parts yields
\begin{align}\label{i1-ibp}
I_{1}(b,\ell)&=\left[u\log(1+ bu^{2})\rho_{\ell}(u)\right]_{0}^{\infty}-\int_{0}^{\infty}\left(\frac{2bu^2}{1+bu^2}+\log(1+bu^2)\right)\rho_{\ell}(u)\, du.
\end{align}
Note that $\rho_{\ell}(t)\to 0$ as $t\to \infty$. This is easily seen for $\rho_1$ since
\begin{equation}\label{rho1-infty}
\rho_1(t)=\frac{\pi}{2}\text{sech}^{2}(\pi t)=\frac{2\pi e^{-2\pi t}}{(1+e^{-2\pi t})^2}\to 0 \hspace{2mm}\text{as}\hspace{2mm} t\to\infty.
\end{equation}
For $\ell\geq 2$, use the definition of $\rho_{\ell}(t)$ in \eqref{convo-1}, and the above asymptotic for $\rho_1$, along 
with Lebesgue's dominated convergence theorem to deduce that $\rho_{\ell}(t)\to 0$ as $t\to \infty$. As $t\to 0$, it is easy
 to see that the densities $\rho_{\ell}(t)$ are finite. 

This implies that the boundary terms in \eqref{i1-ibp} vanish so that
\begin{align}\label{i1-final}
I_{1}(b,\ell)&=-z_{\ell}(b)-2b\int_{0}^{\infty}\rho_{\ell}(u)\frac{d}{db}\log(1+bu^2)\, du\nonumber\\
&=-z_{\ell}(b)-2b\frac{d}{db}z_{\ell}(b),
\end{align}
where differentiation (with respect to $b$) under the integral sign was employed in the last step.

Now consider $I_{2}(b,\ell)$, use integration by parts twice, and note that the boundary terms again vanish, 
thereby giving
\begin{align}
I_{2}(b,\ell)&=\int_{0}^{\infty}\left\{\frac{b(\ell^2+(20-b\ell^2)u^2+12bu^4)}{2(1+bu^2)^2}+2\log(1+bu^2)\right\}\rho_{\ell}(u)\, du\nonumber\\
&=2z_{\ell}(b)+\frac{b}{2}\int_{0}^{\infty}\frac{(\ell^2+(20-b\ell^2)u^2+12bu^4)}{(1+bu^2)^2}\rho_{\ell}(u)\, du.
\end{align}
Next, use the following representation
\begin{equation}
\frac{(\ell^2+(20-b\ell^2)u^2+12bu^4)}{(1+bu^2)^2}=\frac{\ell^2}{(1+bu^2)^2}+\frac{(8-b\ell^2)u^2}{(1+bu^2)^2}+\frac{12u^2}{1+bu^2}
\end{equation}
to rewrite the above expression for $I_{2}(b,\ell)$ in the form
\begin{align}\label{i2-rewritten}
I_{2}(b,\ell)&=2z_{\ell}(b)+\frac{b}{2}\bigg\{\ell^2\int_{0}^{\infty}\frac{\rho_{\ell}(u)}{(1+bu^2)^2}\, du+(8-b\ell^2)\int_{0}^{\infty}\frac{u^2\rho_{\ell}(u)}{(1+bu^2)^2}\, du\nonumber\\
&\quad\quad\quad\quad\quad\quad+12\int_{0}^{\infty}\frac{u^2\rho_{\ell}(u)}{1+bu^2}\, du\bigg\}.
\end{align}
As shown before,
\begin{equation}\label{i2-ibp-1}
\int_{0}^{\infty}\frac{u^2\rho_{\ell}(u)}{1+bu^2}\, du=\frac{d}{db}z_{\ell}(b).
\end{equation}
Since 
\begin{equation*}
\frac{\rho_{\ell}(u)}{1+bu^2}=\rho_{\ell}(u)-b\frac{u^2\rho_{\ell}(u)}{1+bu^2},
\end{equation*}
and $\rho_{\ell}$, being a probability density, satisfies $\int_{-\infty}^{\infty}\rho_{\ell}(u)=1$, it is seen that
\begin{align}
\int_{0}^{\infty}\frac{\rho_{\ell}(u)}{1+bu^2}\, du=\frac{1}{2}-b\int_{0}^{\infty}\frac{u^2\rho_{\ell}(u)}{1+bu^2}\, du=\frac{1}{2}-b\frac{d}{db}z_{\ell}(b).
\end{align}
Differentiation (with respect to $b$) under the integral sign then gives
\begin{align}\label{i2-ibp-2}
\int_{0}^{\infty}\frac{u^2\rho_{\ell}(u)}{(1+bu^2)^2}\, du=b\frac{d^2}{db^2}z_{\ell}(b)+\frac{d}{db}z_{\ell}(b).
\end{align}
Similarly it can be shown that
\begin{align}\label{i2-ibp-3}
\int_{0}^{\infty}\frac{\rho_{\ell}(u)}{(1+bu^2)^2}\, du=\frac{1}{2}-2b\frac{d}{db}z_{\ell}(b)-b^2\frac{d^2}{db^2}z_{\ell}(b).
\end{align}
Now substitute \eqref{i2-ibp-1}, \eqref{i2-ibp-2} and \eqref{i2-ibp-3} in \eqref{i2-rewritten} and simplify to obtain
\begin{align}\label{i2-final}
I_{2}(b,\ell)&=b^2(4-b\ell^2)\frac{d^2}{db^2}z_{\ell}(b)+b\left(10-\frac{3b\ell^2}{2}\right)\frac{d}{db}z_{\ell}(b)+2z_{\ell}(b)+\frac{b\ell^2}{4}.
\end{align}
Then substitute \eqref{i1-final} and \eqref{i2-final} in \eqref{dde-fform} to deduce that
\begin{align}\label{dde-inter}
\ell(\ell+2)z_{\ell+2}(b)&=b^2(4-b\ell^2)\frac{d^2}{db^2}z_{\ell}(b)+2b\left(1-2\ell-\frac{3b\ell^2}{4}\right)\frac{d}{db}z_{\ell}(b)\nonumber\\
&\quad+\ell(\ell+1)z_{\ell}(b)+\frac{b\ell^2}{4}.
\end{align}
Now let $b=(x-\ell/2)^{-2}$ as in Theorem \ref{coro-nice1}, so that defining 
\begin{equation*}
y_{\ell}(x):=z_{\ell}(b),
\end{equation*}
and replacing $\ell$ by $\ell+2$, gives $z_{\ell+2}(\left(x-1-\ell/2\right)^{-2})=y_{\ell+2}(x)$, and hence
\begin{equation}\label{b-shift}
z_{\ell+2}(b)=y_{\ell+2}(x+1).
\end{equation}
A direct computation now gives 
\begin{align}\label{der-y}
\frac{d}{db}z_{\ell}(b)&=-\frac{1}{2}\left(x-\frac{\ell}{2}\right)^{3}y_{\ell}'(x),\nonumber\\
\frac{d^2}{db^2}z_{\ell}(b)&=\frac{1}{4}\left(x-\frac{\ell}{2}\right)^{6}y_{\ell}''(x)+\frac{3}{4}\left(x-\frac{\ell}{2}\right)^{5}y_{\ell}'(x),
\end{align}
where the prime denotes differentiation with respect to $x$. Finally, substitute \eqref{b-shift} and \eqref{der-y} in
 \eqref{dde-inter} to arrive at \eqref{dde}.
\end{proof}

\noindent
\textbf{Remark:} The case $\ell=1$  of Theorem \ref{coro-nice1} was derived in \cite[Equation (2.28)]{dixit-2014a} and an 
elementary proof of the case $\ell=2$ 
is given in the next section. The differential-difference equation \eqref{dde} then produces the  values of 
\begin{equation*}
 \int_{0}^{\infty}  \log(1+ bu^{2}) \, \rho_{\ell}(u) \, du
\end{equation*}
for any $\ell>2$.

\section{The special case $\ell=2$ of the generating function of the N\"{o}rlund polynomials}
\label{sec-gf-nor00}
 
In this section, we present a different proof of Theorem \ref{gen-fun-modber-0} for $\ell = 2$ which was, in fact, 
the genesis of this project. It involves brute force verification of Theorem \ref{coro-nice1} when $\ell=2$. It is then 
used along with the result in \eqref{eval-int}, namely,
\begin{align}\label{eval2}
\text{eval}\{ \log \mathfrak{B}^{( 2 )}(x) \}=\log \left| x - 1 \right| + \int_{0}^{\infty} \log(1 + bu^{2}) \rho_{2}(u) \, du
\end{align}
 with $b=(x-1)^{-2}$, and the special case of \eqref{eval-gen}, namely,
\begin{equation}\label{bns2}
\sum_{n=1}^{\infty} B_{n}^{(2)*} z^{n}=- \frac{1}{2} \log z - \frac{1}{2}  \text{eval} \left\{ \log  \mathfrak{B}^{(2)} \left( z + \frac{1}{z} \right) \right\}.
\end{equation}
The point to illustrate here is that these calculations soon become out of reach 
for large values of $\ell$. In fact, the case $\ell=3$ itself required six different integrals to be evaluated
in order to arrive at Theorem \ref{coro-nice1} through the direct computation of the integral. At the end of the previous section, another way of calculating these
integrals for all $\ell$ through a differential-difference equation was given. However, this being a recursive way, not only does
it not give an explicit formula but also for higher values of $\ell$, evaluating the integrals this way is a cumbersome process. 
These shortcomings are what led us to seek a new representation for $\text{eval}\{ \log \mathfrak{B}^{( \ell )}(x) \}$, namely
 Theorem \ref{eval-mess00}, which gives an explicit formula for these integrals, avoiding messy calculations at the same time.


\begin{proposition}
Let $a \neq 0$ and 
\begin{equation}\label{integral}
I(a) =  \int_{0}^{\infty} \frac{(x \coth x -1) \log( 1 + a^{2}x^{2})}{\sinh^{2}x} \, dx.
\end{equation}
\noindent
Then 
\begin{equation}
I(a) =  -\log c - 1 +  \psi(c) +  c \, \psi'(c),
\end{equation}
\noindent
with $\begin{displaystyle}c = \frac{1}{\pi a} \end{displaystyle}$.
\end{proposition}
\begin{proof}
To evaluate this integral, observe first that 
\begin{equation}
\frac{d}{dx} \left( \coth x - \frac{x}{\sinh^{2}x} \right) = 
\frac{2(x \coth x - 1)}{\sinh^{2} x} 
\end{equation}
\noindent
and write 
\begin{equation}
I(a) = \frac{1}{2} \int_{0}^{\infty} \log( 1 + a^{2}x^{2}) 
\frac{d}{dx} \left( \coth x - \frac{x}{\sinh^{2}x} \right) \, dx.
\end{equation}
\noindent
In order to integrate by parts and guarantee the convergence of the boundary terms, write 
the integral as 
\begin{equation}
I(a) = \frac{1}{2} \int_{0}^{\infty} \log( 1 + a^{2}x^{2}) 
\frac{d}{dx} \left( \coth x - 1 - \frac{x}{\sinh^{2}x} \right) \, dx.
\end{equation}
\noindent
Integrate by parts and verify that the boundary terms vanish to produce 
\begin{equation}
I(a) = a^{2} \left( I_{1}(a)  + I_{2}(a) \right)
\label{form-ia}
\end{equation}
\noindent
with 
\begin{equation}
I_{1}(a) = \int_{0}^{\infty} \frac{x (1  - \coth x) }{1+a^{2}x^{2}} \, dx \text{ and }
I_{2}(a) =  \int_{0}^{\infty} \frac{x^{2} \, dx}{(1+a^{2}x^{2}) \sinh^{2}x}.
\end{equation}
\noindent
The evaluation of $I_{1}(a)$ is described first. Write it as 
\begin{eqnarray*}
I_{1}(a) & = & - \int_{0}^{\infty} e^{-x} \frac{x}{\sinh x} \, \frac{dx}{1+a^{2}x^{2}} \\
& = & - 2 \int_{0}^{\infty} \frac{x \, dx}{(1+a^{2}x^{2})(e^{2x}-1)} \\
& = & \frac{1}{a^{2}} \left[ \psi \left( \frac{1}{\pi a} \right) + \frac{\pi a}{2}+ \log( \pi a) 
\right]
\end{eqnarray*}
\noindent
using Entry $3.415.1$ of \cite{gradshteyn-2015a}:
\begin{equation*}
\int_{0}^{\infty} \frac{x \, dx}{(x^{2}+ \beta^{2}) ( e^{\mu x} - 1) } = 
\frac{1}{2} \left[ \log \left( \frac{\beta \mu}{2 \pi} \right) - 
\frac{\pi}{\beta \mu} - \psi \left( \frac{\beta \mu}{2 \pi} \right) 
\right], \quad \realpart{ \beta} > 0, \,\, \realpart{\mu} > 0,
\end{equation*}
\noindent
where $\psi(z) = \frac{d}{dz} \log \Gamma(z)$ is the digamma function. A direct 
proof of this entry and some generalizations appear in \cite{boros-2003a}.

To evaluate $I_{2}(a)$ write it as 
\begin{eqnarray}
I_{2}(a) & = & 4 \int_{0}^{\infty} \frac{x^{2} e^{2x} \, dx}{(1+a^{2}x^{2})(e^{2x}-1)^{2}} \\
& = & - 2 \int_{0}^{\infty} \frac{x^{2}}{1+a^{2}x^{2}} \frac{d}{dx} \left( \frac{1}{e^{2x}-1} \right) \, dx.
\nonumber
\end{eqnarray}
\noindent
Integration by parts and a simple scaling produces 
\begin{equation}
I_{2}(a) = \frac{4}{a^{4} \pi^{2}} \int_{0}^{\infty} \frac{x \, dx}{(x^{2}+c^{2})^{2} (e^{2 \pi x}-1)}
\end{equation}
\noindent
with $c = 1/(\pi a)$. Entry $3.415.2$ in \cite{gradshteyn-2015a}, established in \cite{boros-2003a}, 
states that 
\begin{equation}
\int_{0}^{\infty} \frac{x \, dx}{(x^{2} + \beta^{2})^{2} (e^{2 \pi x} - 1)} = 
- \frac{1}{8 \beta^{3}} - \frac{1}{4 \beta^{2}} + \frac{1}{4 \beta} \psi'(\beta),
\end{equation}
which gives 
\begin{equation}
I_{2}(a) = - \frac{\pi}{2a} - \frac{1}{a^{2}} + \frac{1}{\pi a^{3}} \psi' \left( \frac{1}{\pi a} \right).
\end{equation}
\noindent
Replacing the values of $I_{1}(a)$ and $I_{2}(a)$ in \eqref{form-ia} gives the result.
\end{proof}

We now obtain Theorem \ref{gen-fun-modber-0} for $\ell=2$ using the above proposition. To that end,
 let $ x=\pi u$ and $a^2=b/\pi^2$ in \eqref{integral} and use Lemma \ref{density-s2} to find that
\begin{equation*}
\int_{0}^{\infty}\rho_2(u)\log(1+bu^2)\, du=\psi\left(\frac{1}{\sqrt{b}}\right)+\frac{1}{\sqrt{b}}\psi'\left(\frac{1}{\sqrt{b}}\right)+\frac{1}{2}\log b-1.
\end{equation*}
The above equation, along with \eqref{eval2} and the fact that $b=(x-1)^{-2}$, yields
\begin{equation*}
{\rm{eval}} \left\{ \log \mathfrak{B}^{(2)}(x) \right\} = 
\psi \left(  \left| x - 1\right| \right)+\left| x - 1\right|\psi' \left(  \left| x - 1\right| \right)-1.
\end{equation*}
Hence for $x\geq 1$, we have
\begin{equation}
{\rm{eval}} \left\{ \log \mathfrak{B}^{(2)}(x) \right\} = 
\psi \left(   x - 1 \right)+ (x - 1)\psi' \left(  x - 1 \right)-1.
\end{equation}
Substitute this in \eqref{bns2} to obtain
\begin{equation}
\sum_{n=1}^{\infty}{B_{n}^{(2)}}^{*}z^n=
-\frac{1}{2}\log z-\frac{1}{2}\left\{\psi\left(z+\frac{1}{z}-1\right)+\left(z+\frac{1}{z}-1\right)\psi'\left(z+\frac{1}{z}-1\right)-1\right\}.
\end{equation}
This completes the proof.

\section{Integrals involving Chebyshev polynomials}
\label{sec-chebyshev}

This section presents the evaluation of some integrals involving the Chebyshev polynomials 
obtained as byproducts of the former results. The proof uses the Binet formulas \eqref{cheby-t-binet} and \eqref{cheby-u-binet} 
for these polynomials. The discussion begins with some preliminary results.

\begin{lemma}
\label{lemma-61}
Let $b>0$. Then 
\begin{equation}
\frac{d^{2\ell}}{du^{2\ell}} \log(1+ bu^{2}) = \frac{2 (-1)^{\ell-1} b^{\ell} (2\ell-1)!}{(1+bu^{2})^{\ell}} 
T_{2\ell} \left( \frac{1}{\sqrt{1+bu^{2}}}\right)
\end{equation}
\noindent
and 
\begin{equation}
\frac{d^{2\ell+1}}{du^{2\ell+1}} \log(1+ bu^{2}) = \frac{2 (-1)^{\ell} b^{\ell+1} (2\ell)!u}{(1+bu^{2})^{\ell+1}} 
U_{2\ell} \left( \frac{1}{\sqrt{1+bu^{2}}} \right).
\end{equation}
\end{lemma}
\begin{proof}
The proof is given for the second formula. The first one can be established by the same 
procedure. Successive differentiation gives 
\begin{equation}
\frac{d^{2\ell+1}}{du^{2\ell+1}} \log( 1 \pm i \sqrt{b}u) = \pm \frac{(-1)^{\ell} i b^{\ell+ 1/2} (2\ell)!}
{(1 \pm i \sqrt{b}u)^{2\ell+1}}.
\end{equation}
\noindent
Hence 
\begin{eqnarray*}
\frac{d^{2\ell+1}}{du^{2\ell+1}} \log(1 + bu^{2})  & = & 
(-1)^{\ell} i b^{\ell+ 1/2} (2\ell)! 
\left\{ \frac{1}{(1 + i \sqrt{b} u)^{2\ell+1}} - 
 \frac{1}{(1 - i \sqrt{b} u)^{2\ell+1}}  \right) \\
 & = & \frac{2 (-1)^{\ell} b^{\ell+1} (2\ell)! u}{(1+bu^{2})^{\ell+1}} 
 U_{2\ell} \left( \frac{1}{\sqrt{1+ bu^{2}}} \right)
 \end{eqnarray*}
 \noindent
 using \eqref{cheby-u-binet}.
\end{proof}

The representation for the densities $\rho_{\ell}(u)$ given by Airault are now used to 
produce some spectacular integrals involving the Chebyshev polynomials.

\begin{theorem}
Let $T_{\ell}(x)$ be the Chebyshev polynomial of the first kind. Define 
\begin{equation}
P_{1}(u,\ell) = \prod_{j=1}^{\ell-1} (u^{2} + j^{2}) \text{ and }
P_{2}(u,\ell) = \prod_{j=1}^{\ell}\left(u^{2} + \left( j - \tfrac{1}{2} \right)^{2} \right)
\end{equation}
\noindent
Then, for $x > \ell$,
\begin{multline*}
\int_{0}^{\infty} \left\{ \frac{u P_{1}(u,\ell)}{\tanh( \pi u)} - u^{2\ell-1} \right\} 
T_{2\ell} \left( \frac{x-\ell}{\sqrt{u^{2} + (x-\ell)^{2}}} \right) 
\frac{du}{(u^{2}+ (x-\ell)^{2})^{\ell}} = \\
(-1)^{\ell} \left( \log(x-\ell) + H_{2\ell-1} - 
\frac{d^{2\ell-1}}{dx^{2\ell-1}} \left\{ \binom{x-1}{2\ell-1} \psi(x-\ell) \right\} \right),
\end{multline*}
\noindent
and for $x > \ell+\tfrac{1}{2}$,
\begin{multline*}
\int_{0}^{\infty} \left\{ \tanh(\pi u)  P_{2}(u,\ell) - u^{2\ell} \right\} 
U_{2\ell} \left( \frac{x-\ell- \tfrac{1}{2}}{\sqrt{u^{2} + (x-\ell - \tfrac{1}{2} )^2 }} \right) 
\frac{u \, du}{(u^{2}+ (x-\ell - \tfrac{1}{2})^{2})^{\ell+1}} = \\
(-1)^{\ell} \left( \log(x-\ell - \tfrac{1}{2}) + H_{2\ell} - 
\frac{d^{2\ell}}{dx^{2\ell}} \left\{ \binom{x-1}{2\ell} \psi(x-\ell- \tfrac{1}{2}) \right\} \right).
\end{multline*}
\end{theorem}
\begin{proof}
The details are given for the second formula.  The expression for the density 
functions in given by Airault in \eqref{airault-1} are written as 
\begin{equation}
\rho_{2\ell+1}(u) = C_{\ell} \left( \frac{d}{du} \right)^{2\ell+1} \left[P_{2}(u,\ell) \tanh( \pi u) \right]
\end{equation}
\noindent 
with $C_{\ell} = (2 (2\ell)!)^{-1}$. Therefore 
\begin{multline}
\int_{0}^{\infty} \rho_{2\ell+1}(u) \log(1+bu^{2}) \, du   =  C_{\ell} 
\int_{0}^{\infty} \log(1+bu^{2}) \left( \frac{d}{du} \right)^{2\ell+1} 
\left[ P_{2}(u,\ell) \tanh(\pi u ) \right] \, du   \label{int-parts} \\
  =  C_{\ell} 
\int_{0}^{\infty} \log(1+bu^{2})  \frac{d}{du} \left[ \left( \frac{d}{du} \right)^{2\ell} 
\left[ P_{2}(u,\ell) \tanh(\pi u ) \right]  \right] \, du.  \nonumber 
\end{multline}
\noindent
In order to integrate by parts, the boundary terms at $+ \infty$ need to be modified. Observe that 
\begin{eqnarray*}
\left( \frac{d}{du} \right)^{2\ell}  \left[ P_{2}(u,\ell) \tanh(\pi u) \right] 
& = & 
\sum_{j=0}^{2\ell} \binom{2\ell}{j} 
\left( \frac{d}{du} \right)^{j} \left[ \tanh(\pi u) \right]
\left( \frac{d}{du} \right)^{2\ell-j} \left[ P_{2}(u,\ell) \right] \\
& = &  (2\ell)!  \tanh(\pi u) +  \\
& & \quad \sum_{j=1}^{2\ell} \binom{2\ell}{j} 
\left( \frac{d}{du} \right)^{j} \left[ \tanh(\pi u) \right]
\left( \frac{d}{du} \right)^{2\ell-j} \left[ P_{2}(u,\ell) \right] 
\end{eqnarray*}
\noindent
The terms coming from derivatives of $\tanh(\pi u)$ in the second sum are polynomials in $\text{sech}^{2}u$, without a 
constant term. The 
terms coming from $P_{2}(u,\ell)$ are polynomials in $u$. It follows that the  whole second sum 
vanishes as $u \to  + \infty$. Then 
\begin{multline*}
\int_{0}^{\infty} \rho_{2\ell+1}(u) \log(1+bu^{2}) \, du     \\
  =  C_{\ell} 
\int_{0}^{\infty} \log(1+bu^{2})  \frac{d}{du} \left[ \left( \frac{d}{du} \right)^{2\ell} 
\left[ P_{2}(u,\ell) \tanh(\pi u ) - u^{2\ell}  \right]  \right] \, du.  \nonumber 
\end{multline*}
\noindent
and now integration by parts gives 
\begin{multline*}
\int_{0}^{\infty} \rho_{2\ell+1}(u) \log(1+bu^{2}) \, du     \\
  =  C_{\ell} 
\int_{0}^{\infty}   \left[ P_{2}(u,\ell) \tanh(\pi u) - u^{2\ell} \right] 
\left( \frac{d}{du} \right)^{2\ell+1} \log(1 + bu^{2})\, du.  \nonumber 
\end{multline*}
\noindent
Now use Theorem \ref{coro-nice1} to evaluate the integral on the left-hand side and Lemma \ref{lemma-61} to 
obtain the result.
\end{proof}

\section{Relations to the Hurwitz and Barnes zeta functions}
\label{sec-hurwitz}

This section expresses the densities $\rho_{\ell}(x)$ in terms of the Hurwitz zeta function. This is the used to produce the closed-form
evaluations of some integrals involving the  Hurwitz zeta function.

\begin{definition}
Let $N \in \mathbb{N}$ and $w, \, s \in \mathbb{C}$ with 
$\realpart{w} > 0, \, \realpart{s} > N$. The \textit{Barnes zeta function} is defined by the series
\begin{equation}
\zeta_{N}(s,w | a_{1}, \cdots, a_{N}) = \sum_{m_{1}, \cdots, m_{N} = 0}^{\infty}
(w + m_{1}a_{1}+ \cdots + m_{N}a_{N})^{-s}.
\end{equation}
\end{definition}

\noindent
This function was introduced in \cite{barnes-1904a} and contains, as the special case $N=1$ and 
$a_{1}=1$, the  \textit{Hurwitz zeta function} 
\begin{equation}
\zeta(s,w) = \sum_{n=0}^{\infty} \frac{1}{(n+w)^{s}}.
\end{equation}
\noindent 
A class of definite integrals connected to $\zeta(s,w)$ was described in 
\cite{espinosa-2002a,espinosa-2002b}.  In particular, the classical identity of Lerch \cite[entry 9.533.3]{gradshteyn-2015a}
\begin{equation}
\frac{d}{dz} \zeta(z,q) \left |_{z=0} \right.= \log \Gamma(q) - \log \sqrt{2 \pi}
\end{equation}
\noindent
gives  the classical evaluation 
\begin{equation}
\int_{0}^{1} \log \Gamma(z) \, dz = \log \sqrt{2 \pi}
\end{equation}
\noindent
given by L.~Euler, as well as 
\begin{multline}
\label{loggammasquared}
\int_{0}^{1} \log^{2} \Gamma(z) \, dz =  
 \frac{\gamma ^2}{12}  + \frac{\pi ^2}{48} + \frac{1}{3}\gamma \lp
+ \frac{4}{3}\left( {\lp } \right)^2 \\
- (\gamma  + 2 \lp )\frac{\zeta '(2)}{\pi ^2} +
\frac{\zeta ''(2)}{2\pi ^2 }.
\end{multline}
\noindent
The corresponding evaluations  for the integrals of $\log^{\ell}\Gamma(z)$, for $\ell =3, \, 4$ are  more complicated and they involve
 multiple-zeta values. In particular, the existence of formulas for $\ell\geq 5$, remains an open problem. See \cite{baileyd-2014a} for details. 
 
 \smallskip

The connection between the Hurwitz zeta function and the densities $\rho_{\ell}(x)$ is based on an integral representation 
of the Barnes zeta function given by S.~N.~M. Ruijsenaars \cite[p.~121]{ruijsenaars-2000a}. Introducing  the notation 
$\begin{displaystyle}A_{M} = \tfrac{1}{2}(a_{1} + \cdots + a_{M})  \end{displaystyle}$,  it is  shown in 
\cite{ruijsenaars-2000a} that if 
$\delta > -A_{M}$,  the Barnes zeta function has the integral representation 
\begin{equation}
\zeta_{M}(s,A_{M}+ \delta + i z | a_{1}, \cdots, a_{M}) = 
\frac{2^{1-M}}{\Gamma(s)} 
\int_{0}^{\infty} \frac{(2y)^{s-1} e^{-2 \delta y}}{\prod_{1 \leq j \leq M}  \sinh( a_{j}y)} e^{-2 i zy} dy
\end{equation}
\noindent
for  $ \realpart{s}>M$ and $\imagpart{z} < A_{M} + \delta$. Now choose 
 $n \in \mathbb{N}$ and consider the special case $\delta=0, \, M=\ell, \, s = \ell+1$ 
and $a_{j} = 1$ for $1 \leq j \leq M$. This yields the identity 
\begin{equation}
\zeta_{\ell}(\ell+1,\tfrac{\ell}{2}  + i z | 1, \cdots, 1) = 
\frac{2}{\ell!} 
\int_{0}^{\infty} e^{-2 i zy} \left( \frac{y}{\sinh y} \right)^{\ell} \, dy.
\end{equation}
\noindent
The next result gives a  new representation for the density $\rho_{\ell}(x)$ in terms of the  Barnes zeta function. The 
proof comes directly from \eqref{form-fn}.

\begin{proposition}
Let 
\begin{equation}
\zeta_{\ell}(m,z) = \zeta_{\ell}(m,z| 1,\cdots,1).
\end{equation}
\noindent
Then the density function $\rho_{\ell}(x)$ is given by
\begin{equation}
\rho_{\ell}(x) = \frac{\ell!}{2 \pi} \left( \zeta_{\ell}(\ell+1, \tfrac{\ell}{2} + i x) +  \zeta_{\ell}(\ell+1, \tfrac{\ell}{2} - i x) \right).
\label{hur-1}
\end{equation}
\end{proposition}

The next representation for the densities $\rho_{\ell}(x)$ comes from a result of J.~Choi \cite[Equation (2.5)]{choijy-1996a}, which 
expresses $\zeta_{\ell}(s,w)$ as a finite linear combination of the Hurwitz zeta function, in the form 
\begin{equation}
\zeta_{\ell}(s,w) = \sum_{j=0}^{\ell-1} p_{\ell,j}(w) \zeta(s-j,w) 
\end{equation}
\noindent
where 
\begin{equation}
p_{\ell,j}(w) = \frac{(-1)^{\ell+1-j}}{(\ell-1)!} \sum_{m=j}^{\ell-1} \binom{m}{j} s(\ell,m+1) w^{m-j},
\end{equation}
\noindent
where $s(\ell,m)$ is the Stirling number of the first kind. Then \eqref{hur-1} leads to 
\begin{multline}
\rho_{\ell}(x)  =  \frac{\ell (-1)^{\ell+1}}{2 \pi} \sum_{j=0}^{\ell-1} (-1)^{j} \sum_{m=j}^{\ell-1} \binom{m}{j} s(\ell,m+1) \\
   \times \left\{ \left( \frac{\ell}{2} + i x \right)^{m-j} \zeta\left( \ell+1-j, \frac{\ell}{2} + i x \right) + 
    \left( \frac{\ell}{2} - i x \right)^{m-j} \zeta\left( \ell+1-j, \frac{\ell}{2} - i x \right)  \right\}.
    \end{multline}
    
It follows that the logarithmic moment can be expressed as 
\begin{multline*}
\int_{0}^{\infty} \rho_{2\ell}(u) \log \left( 1 + \frac{u^{2}}{(x-\ell)^{2}} \right) \, du =  \\
 \frac{2\ell}{\pi} \sum_{j=0}^{\ell-1} (-1)^{j-1} \sum_{m=j}^{2\ell-1} s(2\ell,m+1) \\
  \times   \realpart{ \int_{0}^{\infty} \left\{ (\ell+ i u)^{m-j} \zeta(2\ell+1-j,\ell+ i u )  \right\}  }
 \log \left( 1 + \frac{u^{2}}{(x-\ell)^{2}} \right) \, du. 
 \end{multline*}
    
Now replace $\ell$ by $2\ell$ in \eqref{eval-int} and Theorem \ref{eval-mess00}, equate their right-hand sides, 
and use the above identity to arrive at first of the following two identities. The second one is similarly proved.


\begin{theorem}
Let $\zeta(s,w)$ denote the Hurwitz zeta function and $s(\ell,m)$ the Stirling numbers of the first kind. Define 
\begin{equation*}
z_{\ell}(m,j)(x) = 2 \realpart \int_{0}^{\infty} (\ell + i u)^{m-j} \zeta(2\ell+1-j,\ell + i u) 
 \log \left( 1 + \frac{u^{2}}{(x-\ell)^{2}} \right) \, du
 \end{equation*}
 \noindent
 and 
 \begin{equation*}
Z_{\ell}(m,j)(x) = 2 \realpart \int_{0}^{\infty} (\ell + \tfrac{1}{2} + i u)^{m-j} \zeta(2\ell+2-j,\ell +  \tfrac{1}{2} + i u) 
 \log \left( 1 + \frac{u^{2}}{(x-\ell-\tfrac{1}{2})^{2}} \right) \, du
 \end{equation*}
 \noindent
 Then, for $x > \ell$, 
 \begin{multline*}
 \sum_{j=0}^{2\ell-1} (-1)^{j-1} \sum_{m=j}^{2\ell-1} \binom{m}{j}  s(2\ell,m+1) z_{\ell}(m,j)(x) =  \\
 - \frac{\pi}{\ell} \left( \log(x-\ell) + H_{2\ell-1} - 
 \frac{d^{2\ell-1}}{dx^{2\ell-1}} \left\{ \binom{x-1}{2\ell-1} \psi(x-\ell) \right\} \right),
 \end{multline*}
 \noindent
 and  $x > \ell+\tfrac{1}{2}$,
 \begin{multline*}
 \sum_{j=0}^{2\ell} (-1)^{j} \sum_{m=j}^{2\ell} \binom{m}{j}  s(2\ell+1,m+1) Z_{\ell}(m,j)(x) =  \\
 - \frac{2\pi}{2\ell+1} \left( \log \left( x-\ell - \tfrac{1}{2} \right)  + H_{2\ell} - 
 \frac{d^{2\ell}}{dx^{2\ell}} \left\{ \binom{x-1}{2\ell} \psi \left( x-\ell- \tfrac{1}{2}) \right) \right\} \right).
 \end{multline*}
 \end{theorem}
 
 Inverting these systems of equations to obtain expressions for $Y_{n}(m,j)$ and $Z_{n}(m,j)$ is an 
 open problem.
 
\medskip

\noindent
\textbf{Acknowledgments}.  The authors wish to thank Larry Glasser for discussions on this project. Partial 
support for the work of the third author comes from NSF-DMS 1112656. The first author is a post-doctoral fellow, funded 
in part by the same grant. The work of the C.~Vignat was partially supported by the iCODE Institute, 
Research Project of the Idex Paris-Saclay.


\begin{thebibliography}{10}

\bibitem{airault-2008a}
H.~Airault.
\newblock Hyperbolic measures, moments and coefficients. {A}lgebra on
  hyperbolic functions.
\newblock {\em J. {F}unct. {A}nal.}, 255:2099--2145, 2008.

\bibitem{baileyd-2014a}
D.~H. Bailey, D.~Borwein, and J.~M. Borwein.
\newblock {E}ulerian {L}og-{G}amma integrals and {T}ornheim-{W}itten zeta
  functions.
\newblock {\em The {R}amanujan {J}ournal, to appear}, 2015.

\bibitem{barnes-1904a}
E.~W. Barnes.
\newblock On the theory of the multiple gamma function.
\newblock {\em Trans. Camb. Philos. Soc.}, 19:374--425, 1904.

\bibitem{boros-2003a}
G.~Boros, O.~Espinosa, and V.~Moll.
\newblock On some families of integrals solvable in terms of polygamma and
  negapolygamma functions.
\newblock {\em Integrals {T}ransforms and {S}pecial {F}unctions}, 14:187--203,
  2003.

\bibitem{choijy-1996a}
J.~Choi.
\newblock Explicit formulas for the {B}ernoulli polynomial of order $n$.
\newblock {\em Indian {J}. {P}ure {A}ppl. {M}ath.}, 27:667--674, 1996.

\bibitem{dixit-2014f}
A.~Dixit, A.~Kabza, V.~Moll, and C.~Vignat.
\newblock The integrals in {G}radshteyn and {R}yzhik. {P}art 30: {M}ore
  hyperbolic entries.
\newblock {\em In preparation}, 2015.

\bibitem{dixit-2014a}
A.~Dixit, V.~Moll, and C.~Vignat.
\newblock The {Z}agier modification of {B}ernoulli numbers and a polynomial
  extension. {P}art {I}.
\newblock {\em The Ramanujan Journal}, 33:379--422, 2014.

\bibitem{espinosa-2002a}
O.~Espinosa and V.~Moll.
\newblock On some definite integrals involving the {H}urwitz zeta function.
  {P}art 1.
\newblock {\em The {R}amanujan {J}ournal}, 6:159--188, 2002.

\bibitem{espinosa-2002b}
O.~Espinosa and V.~Moll.
\newblock On some definite integrals involving the {H}urwitz zeta function.
  {P}art 2.
\newblock {\em The {R}amanujan {J}ournal}, 6:449--468, 2002.

\bibitem{gessel-2003a}
I.~Gessel.
\newblock Applications of the classical umbral calculus.
\newblock {\em Algebra Universalis}, 49:397--434, 2003.

\bibitem{gradshteyn-2015a}
I.~S. Gradshteyn and I.~M. Ryzhik.
\newblock {\em Table of {I}ntegrals, {S}eries, and {P}roducts}.
\newblock Edited by D. Zwillinger and V. Moll. Academic Press, New York, 8th
  edition, 2015.

\bibitem{norlund-1924a}
N.~E. N\"{o}rlund.
\newblock {\em Vorlesungen \"{u}ber {D}ifferenzen-{R}echnung}.
\newblock Berlin, 1924.

\bibitem{pitman-2003a}
J.~Pitman and M.~Yor.
\newblock Infinitely divisible laws associated with hyperbolic functions.
\newblock {\em Canad. {J}. {M}ath.}, 55:292--330, 2003.

\bibitem{ruijsenaars-2000a}
S.~N.~M. Ruijsenaars.
\newblock On {B}arnes' multiple zeta and gamma function.
\newblock {\em Adv. Math.}, 156:107--132, 2000.

\bibitem{spanier-1987a}
J.~Spanier and K.~Oldham.
\newblock {\em An atlas of functions}.
\newblock Hemisphere {P}ublishing Co., 1st edition, 1987.

\bibitem{zagier-1998a}
D.~Zagier.
\newblock A modified {B}ernoulli number.
\newblock {\em Nieuw {A}rchief voor {W}iskunde}, 16:63--72, 1998.

\end{thebibliography}
%

\end{document}